\documentclass{amsproc}
\usepackage{amssymb}
\usepackage{amsfonts}

\theoremstyle{plain}

\newtheorem{definition}{Definition}

\newtheorem{lemma}{Lemma}

\newtheorem{proposition}{Proposition}
\newtheorem{remark}{Remark}

\newtheorem{theorem}{Theorem}
\numberwithin{equation}{section}
\input{tcilatex}

\begin{document}
\title[Smoothness of orbital measures]{The smoothness of convolutions of
orbital measures on complex Grassmannian symmetric spaces}
\author{Sanjiv Kumar Gupta}
\address{Dept. of Mathematics and Statistics\\
Sultan Qaboos University\\
P.O.Box 36 Al Khodh 123\\
Sultanate of Oman}
\email{gupta@squ.edu.om}
\author{Kathryn E. Hare}
\address{Dept. of Pure Mathematics\\
University of Waterloo\\
Waterloo, Ont.,~Canada\\
N2L 3G1}
\email{kehare@uwaterloo.ca}
\thanks{This research is supported in part by NSERC 2016-03719 and by Sultan
Qaboos University. The authors thank Acadia University for their hospitality
when this research was done. }
\subjclass[2000]{Primary 43A90, 43A85; Secondary 58C35, 33C50}
\keywords{orbital measure, spherical functions, complex Grassmannian
symmetric space, absolute continuity}
\thanks{This paper is in final form and no version of it will be submitted
for publication elsewhere.}

\begin{abstract}
It is well known that if $G/K$ is any irreducible symmetric space and $\mu
_{a}$ is a continuous orbital measure supported on the double coset $KaK,$
then the convolution product, $\mu _{a}^{k},$ is absolutely continuous for
some suitably large $k\leq \dim G/K$. The minimal value of $k$ is known in
some symmetric spaces and in the special case of groups or rank one
symmetric spaces it has even been shown that $\mu _{a}^{k}$ belongs to the
smaller space $L^{2}$ for some $k$. Here we prove that this $L^{2}$ property
holds for all the compact, complex Grassmanian symmetric spaces, $%
SU(p+q)/S(U(p)\times U(q))$. Moreover, for the orbital measures at a dense
set of points $a$, we prove that $\mu _{a}^{2}\in L^{2}$ (or $\mu
_{a}^{3}\in L^{2}$ if $p=q$).
\end{abstract}

\maketitle

\section{\protect\bigskip Introduction}

Suppose $G$ is a Lie group with Cartan involution $\theta $ and $K$ is the
compact subgroup of $G$ fixed by $\theta $. In a now classical paper, \cite%
{Ra}, Ragozin proved that if the symmetric space, $G/K,$ is irreducible,
then any convolution product of $\dim G/K,$ continuous, $K$-bi-invariant
measures on $G$ is absolutely continuous with respect to the Haar measure on 
$G$. In the case that $G$ is a non-compact group, Ragozin's result was
significantly improved by Graczyk and Sawyer in \cite{GSFunc} who showed
that $rank(G/K)+1$ convolutions would suffice. The building blocks for all $K
$-bi-invariant measures on $G/K$ are the so-called orbital measures, $\mu
_{a},$ supported on the double cosets, $KaK$.\ Building on the work of
Graczyk and Sawyer in \cite{GSFunc}-\cite{GSAust}, the authors in \cite%
{GHJMAA} improved this result by characterizing the convolution products of
orbital measures that are absolutely continuous, for any classical,
non-compact symmetric space.

For the compact symmetric spaces $G_{c}/K$ for $G_{c}$ compact, the authors
in \cite{GHJMAA2} proved that any convolution product of $2rank(G_{c}/K)+1,$
continuous, $K$-bi-invariant measures is absolutely continuous. Any compact
Lie group, $G_{c},$ can be regarded as a compact symmetric space, namely $%
G_{c}\times G_{c}/K$ with $K=\{(g,g):g\in G_{c}\}$. The $K$--bi-invariant,
orbital measures are the $G_{c}$-invariant measures supported on the
conjugacy classes of the group. In this special case, it is known that $\mu
_{a}^{k}$ belongs to the (smaller) space $L^{2}(G_{c})$ if and only if $\mu
_{a}^{k}$ is absolutely continuous (equivalently, belongs to $L^{1}(G_{c})$%
). Moreover, the minimal exponent, $k(a),$ such that $\mu _{a}^{k(a)}\in
L^{2}$ has even been determined for each $a$; see \cite{GHAdv, HJSY, Wr}.
This $L^{1}-L^{2}$ dichotomy was shown to be false in the compact symmetric
space $SU(2)/SO(2)$, although it is still true that $\mu _{a}^{3}\in L^{2}$
for all continuous orbital measures in any rank one symmetric space; see 
\cite{AGP}, \cite{HH}. These $L^{2}$ results were all found by studying the
decay properties of characters or spherical functions.

In \cite{AA}, Al-Hashami and Anchouche studied the analogous problem for the
compact, complex Grassmannian symmetric spaces $G_{c}/K,$ where $%
G_{c}=SU(p+q)$ and $K=S(U(p)\times U(q))$. They proved that for the dense
set of `regular' points $a\in G_{c}$, $\mu _{a}^{k}\in L^{2}$ for a choice
of $k$ which is much smaller than $\dim G_{c}/K$; see (\ref{AA}). They did
this by using the Berezin-Karpelevich formula for the spherical functions
and obtaining rates of decay for these functions at the regular elements.
Using these estimates and the Sobolev embedding theorem, they also proved
that $\mu _{a}^{k}\in C^{s}(G_{c})$ for regular elements $a$ and suitably
large choices of $k,$ depending on $s$.

In this note, we find estimates on the rates of decay for the spherical
functions at \textit{all} points $a\notin $ $N_{G_{c}}(K)$, these being the
points which give rise to the continuous orbital measures $\mu _{a}$.
Obtaining such estimates at non-regular points is quite delicate for then
the singularities of the spherical functions must be understood. With these
estimates we are able to prove that $\mu _{a}^{k}\in L^{2}$ for all $a\notin
N_{G_{c}}(K)$ whenever 
\begin{equation*}
k>\max (p,2(p-q)+3)\text{ if }p>q\text{ and }k>\max (2p,6)\text{ if }p=q.
\end{equation*}%
For the special case of the regular elements $a,$ we show that $\mu
_{a}^{2}\in L^{2}$ if $p>q$ and $\mu _{a}^{3}\in L^{2}$ otherwise. This
significantly improves the work in \cite{AA} as their exponent $k$ was
unbounded in $p$. We also show that $\mu _{a}^{k}\in C^{s}$ for all $a\notin
N_{G_{c}}(K)$ for suitable $k=k(s,a)$.

Although smaller than $\dim G_{c}/K,$ our minimal exponent $k$ is typically
much larger than $2rank(G_{c}/K)+1$. It would be interesting to know what
the sharp exponent is for the $L^{2}$ problem and whether the $L^{1}-L^{2}$
dichotomy holds.

\section{Notation and Preliminaries}

\subsection{Orbital measures and their smoothness}

Suppose $G_{c}/K$ is an irreducible symmetric space, where $G_{c}$ is a
compact Lie group with Cartan involution $\theta $ and $K=\{g\in
G_{c}:\theta (g)=g\}$. Denote its non-compact dual space by $G/K$ and
suppose the Lie algebra of $G$ is $\mathfrak{g}$. We fix the Cartan
decomposition $\mathfrak{g=t\oplus p}$ and the maximal abelian subspace $%
\mathfrak{a}$ of $\mathfrak{p}$ as in \cite{He}. We identify the Lie algebra 
$\mathfrak{g}_{c}$ of $G_{c}$ with the subspace $\mathfrak{g}_{c}\mathfrak{%
=t\oplus ip}$ of the complexified Lie algebra $\mathfrak{g}^{\mathbb{C}}$
and let $A=\exp i\mathfrak{a}\subseteq G$. Always $G_{c}=KAK$. For any $a\in
G_{c},$ the set $KaK$ is called a double coset and has Haar measure zero.

\begin{definition}
Given any $a=\exp (iX)\in A,$ we define the \textbf{orbital measure} $\mu
_{a}$ on $G_{c}$, with support $KaK,$ by%
\begin{equation*}
\int fd\mu _{a}=\int_{K}\int_{K}f(k_{1}ak_{2})dk_{1}dk_{2}
\end{equation*}%
for all continuous functions $f$.
\end{definition}

This probability measure is $K$-bi-invariant. It is continuous (meaning as a
measure on $G_{c}/K$) precisely when $a\notin N_{G_{c}}(K)$ and is purely
singular with respect to Haar measure.

In \cite{Ra}, Ragozin essentially proved the following geometric
characterization.

\begin{proposition}
Let $a_{i}\in A$, $i=1,...,k$. Then $\mu _{a_{1}}\ast \cdot \cdot \cdot \ast
\mu _{a_{k}}$ is absolutely continuous if and only if the product of double
cosets $Ka_{1}K\cdot \cdot \cdot Ka_{k}K$ has non-empty interior in $G_{c}$.
Moreover, if $a_{i}\notin N_{G_{c}}(K)$ and $k\geq \dim (G_{c}/K),$ then $%
\mu _{a_{1}}\ast \cdot \cdot \cdot \ast \mu _{a_{k}}$ is absolutely
continuous.
\end{proposition}

Ragozin used geometric methods to show that $Ka_{1}K\cdot \cdot \cdot
Ka_{k}K $ has non-empty interior for $k=$ $\dim G_{c}/K$ if $a_{i}\notin
N_{G_{c}}(K) $ for all $i$. Notice that if $a\in N_{G_{c}}(K),$ then $\left(
KaK\right) ^{k}=Ka^{k}K$ and hence has Haar measure zero. Thus $\mu _{a}^{k}$
is purely singular to Haar measure for all $k$. Ragozin's geometric
characterization was verified using algebraic methods by the authors in \cite%
{GHJMAA} when they improved upon his result, showing that $%
k=2rank(G_{c}/K)+1 $ suffices. This exponent is close to sharp as there are
continuous orbital measures (in some compact symmetric spaces) with $\mu
_{a}^{2rank(G_{c}/K)-1} $ singular; \cite{GHAdv}.

A measure $\mu $ is absolutely continuous with respect to Haar measure if
and only if its density function (or Radon Nikodym derivative) belongs to $%
L^{1}$. If the density function belongs to the properly smaller space $L^{2}$
(or to $C^{s}),$ we will write $\mu \in L^{2}$ (resp., $C^{s}).$ Ragozin's
geometric approach is not helpful in studying the problem of determining if $%
\mu _{a}^{k}\in L^{2}$. Instead, a harmonic analysis approach has been
taken: estimates are made on the rate of decay of the Fourier transform of
the measure and then the Peter Weyl theorem is invoked. This approach has
been applied very successfully when the symmetric space is a compact group
or a rank one compact symmetric space, and is the approach we will take in
this paper to study the $L^{2}$ problem for the complex Grassmannian
symmetric spaces.

\subsection{The symmetric space $SU(p+q)/S(U(p)\times U(q))$}

For the remainder of this paper we will focus on the case of the complex
Grassmannian symmetric space, $G_{c}/K$, where $G_{c}=SU(p+q)$, $%
K=S(U(p)\times U(q))$ and $p\geq q\geq 2$. This compact symmetric space has
rank $q$ and dimension $2pq$. The non-compact dual space is the symmetric
space $G/K,$ where $G$ has Lie algebra $\mathfrak{g}=su(p,q)$. With Cartan
decomposition $\mathfrak{g=t\oplus p}$, we can take as the maximal abelian
subspace $\mathfrak{a}$ the $(p+q)\times (p+q)$ matrices of the form 
\begin{equation*}
X=X(t_{1},...,t_{q})=%
\begin{bmatrix}
0 & 0 & M \\ 
0 & 0 & 0 \\ 
M^{\prime } & 0 & 0%
\end{bmatrix}%
\end{equation*}%
where $M,M^{\prime }$ are $q\times q$ anti-diagonal matrices, with $%
t_{1},...,t_{q}$ and $t_{q},...,t_{1}$ respectively, on the anti-diagonal.
As above, we let $A=\exp $ $i\mathfrak{a}\subseteq G_{c}$. We can identify $%
X(t_{1},...,t_{q})\in \mathfrak{a}$ by $(t_{1},...,t_{q})\in \mathbb{R}^{q}$.

The restricted roots and the highest spherical weights will be very
important in our work. The positive restricted roots can be taken to be 
\begin{equation*}
\{2e_{k},e_{k},e_{i}\pm e_{j}:1\leq i<j\leq q,1\leq k\leq q\}
\end{equation*}%
(where $\{e_{j}\}$ are the usual basis vectors for $\mathbb{R}^{q})$ with
multiplicities $m_{2e_{k}}=1$, $m_{e_{k}}=2(p-q)$ (so these are not present
if $p=q$) and $m_{e_{i}\pm e_{j}}=2$. The restricted roots act on $\mathfrak{%
a}$ in the natural way. We also view the roots as acting on $A$ by the (well
defined) rule $\alpha (\exp iX)\equiv \alpha (X)\func{mod}\pi $. The reader
is referred to \cite{He} for further background information.

It is well known that the normalizer, $N_{G_{c}}(K)\cap A,$ is characterized
by the property that $\exp iX\in N_{G_{c}}(K)\cap A$ if and only if $\alpha
(X)\equiv 0\func{mod}\pi $ for all restricted roots $\alpha $.

An element $X\in \mathfrak{a}$ or $\exp iX\in A$ is said to be\textbf{\
regular} if $\alpha (X)\neq 0\func{mod}\pi $ for each restricted root $%
\alpha ,$ thus $X=X(t_{1},..,t_{q})$ is regular precisely when $t_{j}\neq
\pm t_{k}\func{mod}\pi $ for all $j\neq k$ and $t_{k}\neq 0,\pi /2$ $\func{%
mod}\pi $ for all $k$. Such elements are dense in $\mathfrak{a}$ or $A$
respectively. When $X$ is regular, $\mu _{\exp iX}^{2}$ is absolutely
continuous (whatever the compact symmetric space); \cite{GHgeneric}.

The highest spherical weights are given by 
\begin{equation*}
\Lambda _{sph}=\left\{ \sum_{j=1}^{q}2m_{j}e_{j}:m_{j}\in \mathbb{N}\text{, }%
m_{1}\geq m_{2}\cdot \cdot \cdot \geq m_{q}\geq 0\right\} .
\end{equation*}%
We denote the spherical function corresponding to $\lambda \in \Lambda
_{sph} $ by $\phi _{\lambda }$. Put 
\begin{eqnarray*}
r &=&p-q+1\text{ and} \\
n_{j} &=&m_{j}+q-j,j=1,...,q
\end{eqnarray*}%
so that $n_{1}>n_{2}>\cdot \cdot \cdot >n_{q}\geq 0$. We denote the
normalized Jacobi polynomials by 
\begin{equation*}
\widetilde{P_{n}}(x)=\frac{P_{n}^{(p-q,0)}(x)}{P_{n}^{(p-q,0)}(1)}
\end{equation*}

According to the Berezin-Karpelevich formula (see \cite{BK} or \cite{Ca}) 
\begin{equation}
\phi _{\lambda }(\exp iX(t_{1},...,t_{q}))=\frac{C_{p,q}\det \left( 
\widetilde{P_{n_{j}}}(\cos 2t_{k})\right) _{j,k=1}^{q}}{\tprod\limits_{1\leq
j<k\leq q}(\cos 2t_{j}-\cos 2t_{k})(n_{j}(n_{j}+r)-n_{k}(n_{k}+r))}
\label{formula}
\end{equation}%
where%
\begin{equation*}
C_{p,q}=2^{q(q-1)/2}\prod_{j=1}^{q-1}j!(j+p-q)^{q-j}
\end{equation*}
and the quotient should be understood in the limiting sense if some $\cos
2t_{j}=\cos 2t_{k}$. Of course, this situation occurs precisely if some $%
t_{j}\equiv \pm t_{k}$ mod $\pi $.

An elementary, but useful, observation is that 
\begin{equation}
n_{j}(n_{j}+r)-n_{k}(n_{k}+r)\geq n_{j}+r\text{ when }j<k.
\label{njproperty}
\end{equation}

\section{Decay of Spherical Functions}

The objective of this section is to obtain estimates on the decay of the
spherical functions. We will find estimates that hold for all $a\notin
N_{G_{c}}(K)$ and better estimates for the regular elements.

\begin{theorem}
\label{decay}Suppose $\lambda =\sum_{j=1}^{q}2m_{j}e_{j}$ is a highest
spherical weight and $n_{j}=m_{j}+q-j$.

(i) If $\exp iX\in A\diagdown N_{G_{c}}(K),$ then%
\begin{equation*}
\left\vert \phi _{\lambda }(\exp iX)\right\vert \leq C\left\{ 
\begin{array}{cc}
\prod_{j=1}^{q-1}(n_{j}+1)^{-1} & \text{if }p>q \\ 
\prod_{j=1}^{q-1}(n_{j}+1)^{-1/2} & \text{if }p=q%
\end{array}%
\right.
\end{equation*}%
where $C$ is a constant that depends on $p,q$ and $X$, but not $\lambda $.

(ii) \label{reg}If $\exp iX$ is regular, then there is a constant $%
C=C(p,q,X) $ such that 
\begin{equation*}
\left\vert \phi _{\lambda }(\exp iX)\right\vert \leq
C\prod_{j=1}^{q}(n_{j}+1)^{-p+j-1/2}.
\end{equation*}
\end{theorem}

\begin{remark}
Note that when $X=X(t_{1},...,t_{q})$ is regular, $\cos 2t_{j}\neq \cos
2t_{k}$ for any $j\neq k$. This is very significant as it means we do not
have the complication of having to understand $\phi _{\lambda }(\exp iX)$
through the limiting process. In \cite{AA}, the decay in $\phi _{\lambda
}(\exp iX)$ was studied for this special case.\footnote{%
The authors do not make this assumption explicit in the statement of their
theorem, but the properties of a regular element are used in their proof.}
They obtained the bound, 
\begin{equation*}
\left\vert \phi _{\lambda }(\exp iX)\right\vert \leq
C\prod_{j=1}^{q}(n_{j}+1)^{-p+q/2}.
\end{equation*}
\end{remark}

In our proof, the constants $C$ which appear may vary from one occurrence to
another, but will always be independent of $\lambda $. We will frequently
write $\phi _{\lambda }(X)$ as shorthand for $\phi _{\lambda }(\exp iX)$.
When we say $f\sim g$ for functions $f,g$ defined on $\mathbb{N}$, we mean
there are constants $A,B>0$ such that $Ag(n)\leq f(n)\leq Bg(n)$ for all $n.$

We begin by collecting useful facts about Jacobi polynomials.

\begin{lemma}
\label{basic}The following are well known facts about Jacobi polynomials:

(i) $P_{n}^{(a,b)}(1)=\binom{n+a}{n}\sim (n+1)^{a};$ $%
P_{n}^{(a,b)}(-1)=(-1)^{n}\binom{n+b}{n};$

(ii) $\left\vert P_{n}^{(a,b)}(x)\right\vert \leq \frac{C_{a,b,x}}{\sqrt{n+1}%
}$ when $x\neq \pm 1;$

(iii) $\left\vert P_{n}^{(a,0)}(x)\right\vert \leq C_{x,a}\left\vert
P_{n}^{(a,0)}(1)\right\vert $ so $\left\vert \widetilde{P_{n}}(x)\right\vert
\leq C_{x,p,q};$\texttt{\ }

(iv) $P_{0}^{(a,b)}(x)=0$ for all $x;$

(v) $\frac{d}{dx}P_{n}^{(a,b)}(x)=\frac{1}{2}(n+a+b+1)P_{n-1}^{(a+1,b+1)}(x)$
if $n\neq 0.$
\end{lemma}

We will first prove part (ii) of the Theorem, the special case when $X$ is
regular.

\begin{proof}
\lbrack of Theorem(ii)] Assume $X=X(t_{1},...,t_{q})$ is regular. Then $%
e_{i}\pm e_{j}(X)=t_{j}\pm t_{k}\neq 0\func{mod}\pi $ for any $j\neq k$ and
that means $\cos 2t_{j}\neq \cos 2t_{k}$ for any $j\neq k$. Furthermore, $%
2e_{k}(X)=2t_{k}\neq 0\func{mod}\pi $ and that implies $\cos (2t_{k})\neq
\pm 1$ for any $k$. This latter fact ensures that

\begin{equation*}
\left\vert \widetilde{P_{n}}(\cos (2t_{k}))\right\vert =\left\vert \frac{%
P_{n}^{(p-q,0)}(\cos (2t_{k}))}{P_{n}^{(p-q,0)}(1)}\right\vert \leq \frac{C}{%
\left( n_{j}+1\right) ^{p-q+1/2}}.
\end{equation*}%
Consequently, formula (\ref{formula}) implies

\begin{eqnarray*}
\left\vert \phi _{\lambda }(X)\right\vert &\leq &\frac{C}{%
\tprod\limits_{1\leq j<k\leq
q}(n_{j}(n_{j}+r)-n_{k}(n_{k}+r))\tprod\limits_{1\leq j\leq
q}(n_{j}+1)^{p-q+1/2}} \\
&\leq &\frac{C}{\tprod\limits_{j=1}^{q-1}(n_{j}+r)^{q-j}\tprod%
\limits_{j=1}^{q}(n_{j}+1)^{p-q+1/2}},
\end{eqnarray*}%
as claimed.
\end{proof}

For general $X$ we must consider the possibility that the formula (\ref%
{formula}) for $\phi _{\lambda }(X)$ has to be understood through the
limiting process. The next several lemmas will help with this. The proof of
the first lemma also introduces a reduction technique that will be
frequently used throughout the remainder of the proof of the theorem.

\begin{lemma}
\label{e1+e2}Fix $X=X(t_{1},...,t_{q})$ and suppose there is an index $k_{0}$
such that $\cos 2t_{k_{0}}\neq $ $\cos 2t_{k}$ for any $k\neq k_{0}$. Then 
\begin{equation*}
\left\vert \phi _{\lambda }(X(t_{1},...,t_{q}))\right\vert \leq
C\prod_{j=1}^{q-1}(n_{j}+1)^{-1}.
\end{equation*}
\end{lemma}

\begin{proof}
First, suppose $q=2$. In this case, 
\begin{equation*}
\left\vert \phi _{\lambda }(X(t_{1},t_{2}))\right\vert =\frac{C\left\vert
\det M\right\vert }{\left\vert \cos 2t_{1}-\cos 2t_{2}\right\vert
(n(n_{1}+r)-n_{2}(n_{2}+r))|}
\end{equation*}%
where from Lemma \ref{basic}(iii) we see that $M$ is a $2\times 2$ matrix
with entries bounded independent of the choice of $\lambda $. As $\cos
2t_{1}\neq \cos 2t_{2},$ it follows that%
\begin{equation*}
\left\vert \phi _{\lambda }(X(t_{1},t_{2}))\right\vert \leq \frac{C}{n_{1}+r}%
\leq C(n_{1}+1)^{-1}
\end{equation*}%
as we desired to show.

Now assume $q>2$. Expand the determinant in (\ref{formula}) along column $%
k_{0}$ to obtain%
\begin{eqnarray*}
\left\vert \det \left( \widetilde{P_{n_{j}}}(\cos 2t_{k})\right)
_{j,k=1}^{q}\right\vert &=&\left\vert \sum_{i=1}^{q}(-1)^{i}\widetilde{%
P_{n_{i}}}(\cos 2t_{k_{0}})\det \left( \widetilde{P_{n_{j}}}(\cos
2t_{k})\right) _{j\neq i,k\neq k_{0}}\right\vert \\
&\leq &q\max_{i}\left\vert \widetilde{P_{n_{i}}}(\cos 2t_{k_{0}})\det \left( 
\widetilde{P_{n_{j}}}(\cos 2t_{k})\right) _{j\neq i,k\neq k_{0}}\right\vert .
\end{eqnarray*}%
Assume the maximum occurs at index $i=j_{0}$. Then 
\begin{equation*}
\left\vert \phi _{\lambda }(X(t_{1},...,t_{q}))\right\vert \leq CI_{1}\cdot
I_{2}
\end{equation*}%
where%
\begin{equation*}
I_{1}=\frac{\left\vert \widetilde{P_{n_{j_{0}}}}(\cos 2t_{k_{0}})\right\vert 
}{\tprod\limits_{k\neq k_{0}}\left\vert \cos 2t_{k}-\cos
2t_{k_{0}}\right\vert \tprod\limits_{j\neq j_{0}}\left\vert
n_{j}(n_{j}+r)-n_{j_{0}}(n_{j_{0}}+r)\right\vert }
\end{equation*}%
and 
\begin{equation*}
I_{2}=\frac{\left\vert \det \left( \widetilde{P_{n_{j}}}(\cos 2t_{k})\right)
_{j\neq j_{0},k\neq k_{0}}\right\vert }{\tprod\limits_{\substack{ 1\leq
j<k\leq q  \\ j,k\neq k_{0}}}|\cos 2t_{j}-\cos 2t_{k}|\tprod\limits 
_{\substack{ 1\leq j<k\leq q  \\ j,k\neq j_{0}}}%
(n_{j}(n_{j}+r)-n_{k}(n_{k}+r))}
\end{equation*}%
(with $I_{2}$ understood in the limiting sense if some $\cos 2t_{j}=\cos
2t_{k}).$

As $\cos 2t_{k}\neq \cos 2t_{k_{0}}$ when $k\neq k_{0}$, applying property (%
\ref{njproperty}) gives%
\begin{eqnarray*}
I_{1} &\leq &\frac{C}{\tprod\limits_{j\neq j_{0}}\left\vert
n_{j}(n_{j}+r)-n_{j_{0}}(n_{j_{0}}+r)\right\vert } \\
&\leq &\frac{C}{(n_{1}+r)\cdot \cdot \cdot
(n_{j_{0}-1}+r)(n_{j_{0}}+r)^{q-j_{0}}}\leq
C\tprod\limits_{j=1}^{q-1}(n_{j}+r)^{-1}
\end{eqnarray*}%
since $n_{1}>n_{2}>\cdot \cdot \cdot >n_{q}.$

In order to bound $I_{2},$ we will use a reduction argument. Consider the
symmetric space $G^{\prime }/K^{\prime }=SU((p-1)+(q-1))/S(U(p-1)\times
U(q-1))$ where $p-1\geq q-1\geq 2,$ and the spherical representation $%
\lambda ^{\prime }=\sum_{i=1}^{q-1}2m_{i}^{\prime }e_{i}$ where $e_{i}$ are
the standard basis vectors for $\mathbb{R}^{q-1}$ and 
\begin{equation*}
m_{i}^{\prime }=\left\{ 
\begin{array}{cc}
n_{i}-(q-1)+i & \text{for }i=1,...,j_{0}-1 \\ 
n_{i+1}-(q-1)+i & \text{for }i=j_{0},...,q-1%
\end{array}%
\right. .
\end{equation*}%
This choice is made so that $n_{i}^{\prime }=m_{i}^{\prime }+(q-1)-i$
satisfies the conditions $n_{i}^{\prime }=n_{i}$ for $i=1,...,j_{0}-1$ and $%
n_{i}^{\prime }=n_{i+1}$ for $i=j_{0},...,q-1$. Let $t^{\prime }=(t_{1},...,%
\widehat{t_{k_{0}}},...,t_{q})$ belong to $\mathfrak{a}^{\prime }$ for $%
G^{\prime }/K^{\prime }$ (here the notation $\widehat{\cdot }$ means the
element is not present) and $\exp iX^{\prime }\in A^{\prime }\subseteq
G^{\prime }$ for $X^{\prime }=X^{\prime }(t^{\prime })$. Notice that for any 
$x,$ 
\begin{equation*}
\widetilde{P_{n}}(x)=\frac{P_{n}^{((p-1)-(q-1),0)}(x)}{%
P_{n}^{((p-1)-(q-1),0)}(1)},
\end{equation*}%
so $I_{2}=C\left\vert \phi _{\lambda ^{\prime }}(X^{\prime })\right\vert
\leq C$ since all spherical functions are bounded by $1$. Thus%
\begin{equation*}
\left\vert \phi _{\lambda }(X(t_{1},...,t_{q}))\right\vert \leq CI_{1}\cdot
I_{2}\leq C\tprod\limits_{j=1}^{q-1}(n_{j}+r)^{-1}
\end{equation*}
\end{proof}

\begin{lemma}
\label{e1+e2(2)}Suppose (upon a suitable reordering of coefficients, if
necessary) $X=X(b_{1},...,b_{s})$ where each $%
b_{j}=(t_{1}^{(j)},...,t_{L_{j}}^{(j)})$ with $\cos (2t_{i}^{(j)})=\cos
(2t_{1}^{(j)})$ for all $i=1,...,L_{j},$ $\cos (2t_{1}^{(i)})\neq \cos
(2t_{1}^{(j)})$ if $i\neq j$, $L_{j}\geq 2$ for all $j$ and $s\geq 2$. Then%
\begin{equation*}
\left\vert \phi _{\lambda }(X)\right\vert \leq
C\tprod\limits_{j=1}^{q}(n_{j}+r)^{-1}.
\end{equation*}
\end{lemma}

\begin{proof}
We will give the details for $s=2,$ but it will be clear how the method
generalizes. Thus assume $X=X(t_{1},...,t_{L_{1}},t_{L_{1}+1},...,t_{q})$
where $\cos (2t_{1})=\cos (2t_{j})$ for $j=2,...,L_{1},$ $\cos (2t_{q})=\cos
(2t_{k})$ for $k=L_{1}+1,...,q-1$, $\cos (2t_{1})\neq \cos (2t_{q})$ and $%
L_{1},L_{2}=q-L_{1}\geq 2$.

A basic property of the determinant is that 
\begin{equation*}
\det \left( \widetilde{P_{n_{j}}}(\cos 2t_{k})\right)
_{j,k=1}^{q}=\sum_{\sigma }c_{\sigma }A_{\sigma }B_{\sigma }
\end{equation*}%
where the sum is over all choices, $\sigma ,$ of $L_{1}$ indices, say $%
\sigma :j_{1}<\cdot \cdot \cdot <j_{L_{1}},$ $A_{\sigma }$ is the
determinant of the $L_{1}\times L_{1}$ matrix $\left( \widetilde{%
P_{n_{j_{i}}}}(\cos 2t_{k})\right) _{i,k=1}^{L_{1}}$, $B_{\sigma }$ is the
determinant of the $L_{2}\times L_{2}$ submatrix of $\left( \widetilde{%
P_{n_{j}}}(\cos 2t_{k})\right) $ formed by the remaining rows and columns
(the remaining columns being columns $L_{1}+1,...,q$), and $c_{\sigma }$ is
a suitable choice of $\pm 1$. This gives the bound 
\begin{equation*}
\left\vert \phi _{\lambda }(X)\right\vert \leq C\sum_{\sigma }I_{\sigma
}J_{\sigma }K_{\sigma },
\end{equation*}%
where%
\begin{equation*}
I_{\sigma }=\frac{\left\vert A_{\sigma }\right\vert }{\tprod\limits_{1\leq
j<k\leq L_{1}}|\cos 2t_{j}-\cos 2t_{k}|\tprod\limits_{1\leq i<l\leq
L_{1}}(n_{j_{i}}(n_{j_{i}}+r)-n_{j_{l}}(n_{j_{l}}+r))}
\end{equation*}

\begin{equation*}
J_{\sigma }=\frac{\left\vert B_{\sigma }\right\vert }{\tprod\limits_{L_{1}+1%
\leq j<k\leq q}|\cos 2t_{j}-\cos 2t_{k}|\tprod\limits_{\substack{ 1\leq
j<k\leq q  \\ j,k\neq j_{1},...,j_{L_{1}}}}(n_{j}(n_{j}+r)-n_{k}(n_{k}+r))},
\end{equation*}%
and 
\begin{equation*}
K_{\sigma }=\frac{1}{\tprod\limits_{\substack{ 1\leq j\leq L_{1}  \\ %
L_{1}+1\leq k\leq 1}}|\cos 2t_{j}-\cos 2t_{k}|\tprod\limits_{\substack{ %
1\leq i\leq L_{1}  \\ k\neq j_{1},...,j_{L_{1}}}}%
(n_{j_{i}}(n_{j_{i}}+r)-n_{k}(n_{k}+r))}.
\end{equation*}%
(As usual, $I_{\sigma },J_{\sigma }$ should be understood in the limiting
sense.)

Since 
\begin{equation*}
\left\vert n_{j_{i}}(n_{j_{i}}+r)-n_{k}(n_{k}+r)\right\vert \geq \max
(n_{j}+r,n_{k}+r)\geq \sqrt{(n_{j}+r)(n_{k}+r)},
\end{equation*}%
it follows that 
\begin{eqnarray*}
&&\tprod\limits_{\substack{ (i,k):1\leq i\leq L_{1}  \\ k\neq
j_{1},...,j_{L_{1}}}}(n_{j_{i}}(n_{j_{i}}+r)-n_{k}(n_{k}+r))\geq
\tprod\limits_{\substack{ (i,k):1\leq i\leq L_{1}  \\ k\neq
j_{1},...,j_{L_{1}}}}\sqrt{(n_{j_{i}}+r)(n_{k}+r)} \\
&=&\prod_{j\in \{j_{1},...,j_{L_{1}}\}}(n_{j}+r)^{L_{2}/2}\prod_{k\notin
\{j_{1},...,j_{L_{1}}\}}(n_{k}+r)^{L_{1}/2} \\
&\geq &\prod_{i=1}^{q}(n_{i}+r)^{\min (L_{1},L_{2})/2}\geq
\prod_{i=1}^{q}(n_{i}+r).
\end{eqnarray*}%
As $\cos 2t_{j}\neq \cos 2t_{k}$ whenever $j\leq L_{1}$ and $K\geq L_{1}+1$,
we see that 
\begin{equation*}
K_{\sigma }\leq C\prod_{i=1}^{q}(n_{i}+r)^{-1}.
\end{equation*}

Similar to the previous lemma, $I_{\sigma }$ and $J_{\sigma }$ are, up to a
constant, the modulus of spherical functions on Grassmanian symmetric spaces
of rank $L_{1}$ and $L_{2}$ respectively. Indeed, $I_{\sigma }=\left\vert
\phi _{\lambda ^{\prime }}(X^{\prime })\right\vert $ on $%
SU(p+q)/S(U(p+q-L_{1})\times U(L_{1}))$ where $\lambda ^{\prime }$ is
identified with $(n_{j_{1}},...,n_{L_{1}})$ and $X^{\prime }=X^{\prime
}(t_{1},...,t_{L_{1}}),$ and similarly for $J_{\sigma }$. Hence both are
bounded and therefore $\phi _{\lambda }(X)$ has the claimed bound.
\end{proof}

\begin{lemma}
\label{2e1}Suppose $X=X(t_{1},...,t_{q})$ where $\cos (2t_{j})=\cos (2t_{k})$
for all $j,k,$ but $\cos (2t_{j})\neq \pm 1$. Then 
\begin{equation*}
\left\vert \phi _{\lambda }(X)\right\vert \leq
C\tprod\limits_{j=1}^{q-1}(n_{j}+r)^{-p+q-1/2}(n_{q}+1)^{-p+q}.
\end{equation*}
\end{lemma}

\begin{proof}
Here we must understand $\phi _{\lambda }(X)$ as 
\begin{equation*}
\lim_{\substack{ x_{i}\rightarrow \cos (2t_{i})  \\ i=1,...,q}}\frac{\det
\left( \widetilde{P_{n_{j}}}(x_{k})\right) _{j,k=1,...,q}}{%
\tprod\limits_{1\leq j<k\leq q}(x_{j}-x_{k})\tprod\limits_{1\leq j<k\leq
q}(n_{j}(n_{j}+r)-n_{k}(n_{k}+r))}
\end{equation*}%
where $x_{i}$ are distinct and not equal to $\cos (2t_{i})$. According to 
\cite{Hua} (see also \cite[Lemma 4.1]{Ho}), this limit is equal to%
\begin{equation}
\phi _{\lambda }(X)=C\frac{\det \left( \widetilde{P_{n_{j}}}^{(k-1)}(\cos
(2t_{1}))\right) _{j,k=1,...,q}}{\tprod\limits_{1\leq j<k\leq
q}(n_{j}(n_{j}+r)-n_{k}(n_{k}+r))}  \label{0}
\end{equation}%
where $C=C(q)=(-1)^{q(q-1)/2}/(q-1)!$ and $\widetilde{P_{n_{j}}}^{(k-1)}$
means the $(k-1)^{\prime }st$ derivative of $\widetilde{P_{n_{j}}}$.
Applying Lemma \ref{basic}(v) repeatedly, we see that 
\begin{equation*}
\widetilde{P_{n}}^{(k-1)}(x)=\frac{2^{-(k-1)}(n+p-q+1)\cdot \cdot \cdot
(n+p-q+k-1)P_{n-(k-1)}^{(p-q+k-1,k-1)}(x)}{P_{n}^{(p-q,0)}(1)},
\end{equation*}%
where we recall that $P_{m}^{(a,b)}(x)=0$ if $m\leq 0$. So the main task is
to bound%
\begin{equation*}
\left\vert \det \left( (n_{j}+p-q+1)\cdot \cdot \cdot
(n_{j}+p-q+k-1)P_{n_{j}-(k-1)}^{(p-q+k-1,k-1)}(\cos (2t_{1})\right)
_{j,k}\right\vert .
\end{equation*}

Using the permutation method for taking derivatives, one can easily see this
determinant is bounded in modulus by 
\begin{equation}
q!\max_{\pi }\left\vert \prod_{j=1}^{q}(n_{j}+p-q+1)\cdot \cdot \cdot
(n_{j}+p-q+\pi (j)-1)P_{n_{j}-(\pi (j)-1)}^{(p-q+\pi (j)-1,\pi (j)-1)}(\cos
(2t_{1}))\right\vert  \label{1}
\end{equation}%
where the maximum is taken over all permutations $\pi $ of $\{1,...,q\}$.

If all $n_{j}\leq q-1$, this expression is clearly bounded (independently of
the choice $(n_{1},...,n_{q})$), so assume otherwise, say%
\begin{equation*}
n_{q}<\cdot \cdot \cdot <n_{s}\leq q-1<n_{s-1}<\cdot \cdot \cdot <n_{1},
\end{equation*}%
for some $s=1,...,q$. Put $s=q+1$ if all $n_{j}>q-1$.

Now, $j\leq s-1$ implies $n_{j}>q-1\geq \pi (j)-1$ and thus $n_{j}-(\pi
(j)-1)\geq (n_{j}+1)/(q+1)$. Since $\cos (2t_{1})\neq \pm 1,$ Lemma \ref%
{basic}(ii) together with this observation yields the bound 
\begin{equation*}
\left\vert P_{n_{j}-(\pi (j)-1)}^{(p-q+\pi (j)-1,\pi (j)-1)}(\cos
2t_{1})\right\vert \leq \frac{C(t_{1},p,q)}{\sqrt{n_{j}+1}}
\end{equation*}%
for these $j$. Moreover, $n_{j}+p-q+\pi (j)-1\leq C_{p,q}(n_{j}+1)$ for
these $j$.

Since always $P_{n}^{(a,b)}(\cos 2t_{1})$ is uniformly bounded over $n$
(with constant depending on $t_{1},a,b)$ it follows that the terms 
\begin{equation*}
(n_{j}+p-q+1)\cdot \cdot \cdot (n_{j}+p-q+\pi (j)-1)P_{n_{j}-(\pi
(j)-1)}^{(p-q+\pi (j)-1,\pi (j)-1)}(\cos (2t_{1}))
\end{equation*}%
are uniformly bounded when $n_{j}\leq q-1$. Consequently, the expression in (%
\ref{1}) for any fixed $\pi $ is bounded by%
\begin{equation*}
C\prod_{j=1}^{s-1}\frac{(n_{j}+p-q+1)\cdot \cdot \cdot (n_{j}+p-q+\pi (j)-1)%
}{\sqrt{n_{j}+1}}\leq C\prod_{j=1}^{s-1}(n_{j}+1)^{\pi (j)-1-1/2}
\end{equation*}%
where $C$ does not depend on the choice of $n_{j}$.

Since the terms $n_{j}$ are strictly increasing, this is maximized with the
permutation $\pi $ that maps $j$ to $q-j+1$. Hence 
\begin{equation*}
(\ref{1})\leq C\prod_{j=1}^{s-1}(n_{j}+1)^{q-j-1/2}.
\end{equation*}

Recalling the formula (\ref{0}), we have 
\begin{eqnarray*}
\left\vert \phi _{\lambda }(X)\right\vert &\leq &\frac{%
C\prod_{j=1}^{s-1}(n_{j}+1)^{q-j-1/2}}{\tprod\limits_{1\leq j<k\leq
q}(n_{j}(n_{j}+r)-n_{k}(n_{k}+r))\prod_{j=1}^{q}P_{n_{j}}^{(p-q,0)}(1)} \\
&\leq &\frac{C\prod_{j=1}^{s-1}(n_{j}+1)^{q-j-1/2}}{%
\prod_{j=1}^{q-1}(n_{j}+1)^{q-j}\prod_{j=1}^{q}(n_{j}+1)^{p-q}}=\frac{%
C\prod_{j=1}^{s-1}(n_{j}+1)^{q-j-1/2}}{\prod_{j=1}^{q}(n_{j}+1)^{p-j}}.
\end{eqnarray*}

When $s=q+1$ (all $n_{j}>q-1$), this simplifies to $%
C\prod_{j=1}^{q}(n_{j}+1)^{-p+q-1/2},$ which is even smaller than the valued
stated in the lemma. Otherwise, $s-1\leq q-1$ and for $j\leq q-1$, $%
q-j-1/2>0,$ thus%
\begin{equation*}
\prod_{j=1}^{s-1}(n_{j}+1)^{q-j-1/2}\leq
\prod_{j=1}^{q-1}(n_{j}+1)^{q-j-1/2}.
\end{equation*}%
We immediately deduce that $\phi _{\lambda }(X)$ is bounded as stated in the
lemma.
\end{proof}

\begin{lemma}
\label{e1}Suppose $X=(t_{1},...,t_{q})$ where $\cos (2t_{j})=-1$ for all $j$%
. Then 
\begin{equation*}
\left\vert \phi _{\lambda }(X)\right\vert \leq
C\tprod\limits_{j=1}^{q}(n_{j}+r)^{-p+q}.
\end{equation*}
\end{lemma}

\begin{remark}
We note that when $p=q,$ $\exp iX\in N_{G_{c}}(K)$ for such $X$.
\end{remark}

\begin{proof}
Similar arguments to the previous proof show that%
\begin{eqnarray*}
\phi _{\lambda }(X) &=&\lim_{\substack{ x_{i}\rightarrow -1  \\ x_{i}\text{
distinct}}}\frac{\det \left( \widetilde{P_{n_{j}}}(x_{k})\right)
_{j,k=1,...,q}}{\tprod\limits_{1\leq j<k\leq
q}(x_{j}-x_{k})\tprod\limits_{1\leq j<k\leq q}(n_{j}(n_{j}+r)-n_{k}(n_{k}+r))%
} \\
&=&C\frac{\det \left( \widetilde{P_{n_{j}}}^{(k-1)}(-1)\right) _{j,k=1,...,q}%
}{\tprod\limits_{1\leq j<k\leq q}(n_{j}(n_{j}+r)-n_{k}(n_{k}+r))}
\end{eqnarray*}

In this situation, bounding the determinant by a multiple of the largest
term in the sum that arises from the permutation method for calculating the
determinant will not give us a good enough estimate. We will actually
directly compute the determinant, instead. As before, we have%
\begin{eqnarray*}
&&\det \left( \widetilde{P_{n_{j}}}^{(k-1)}(-1)\right) \\
&=&\frac{\det (c_{k}(n_{j}+p-q+1)\cdot \cdot \cdot
(n_{j}+p-q+k-1)P_{n_{j}-(k-1)}^{(p-q+k-1,k-1)}(-1))_{j,k}}{%
\prod_{j=1}^{q}P_{n_{j}}^{(p-q,0)}(1)}.
\end{eqnarray*}%
From Lemma \ref{basic}(i),%
\begin{eqnarray*}
P_{n_{j}-(k-1)}^{(p-q+k-1,k-1)}(-1) &=&(-1)^{n_{j}-(k-1)}\binom{n_{j}}{%
n_{j}-(k-1)} \\
&=&\frac{(-1)^{n_{j}-(k-1)}n_{j}\cdot \cdot \cdot (n_{j}-(k-1)+1)}{(k-1)!},
\end{eqnarray*}%
thus we need to evaluate the determinant of the matrix whose $j,k$ entry is
given by%
\begin{equation*}
c_{k}^{\prime }(n_{j}+p-q+1)\cdot \cdot \cdot
(n_{j}+p-q+k-1)(-1)^{n_{j}-(k-1)}n_{j}\cdot \cdot \cdot (n_{j}-(k-1)+1).
\end{equation*}%
In other words, we want to find 
\begin{equation*}
\det \left( (-1)^{n_{1}+\cdot \cdot \cdot
+n_{q}}\prod_{k=1}^{q}c_{k}^{\prime }(-1)^{k-1}M_{j,k}\right) _{j,k}
\end{equation*}%
where $M_{j,1}=1$ and for $k>1,$ 
\begin{eqnarray*}
M_{j,k} &=&\prod_{i=0}^{k-2}(n_{j}+p-q+1+i)(n_{j}-i) \\
&=&\prod_{i=0}^{k-2}(n_{j}+r+i)(n_{j}-i)=%
\prod_{i=0}^{k-2}(n_{j}(n_{j}+r)-i(r+i)).
\end{eqnarray*}

Put $M_{j,1}(x)=1$ and $M_{j,k}(x)=\prod_{\ell =0}^{k-2}(x_{j}-\ell (r+\ell
))$ and consider the multi-variable polynomial 
\begin{equation*}
F(x_{1},...,x_{q})=\det (M_{j,k}(x))_{j,k}.
\end{equation*}%
Our desire is to compute $F(x)$ at $x=(x_{j})_{j=1}^{q}$ with $%
x_{j}=n_{j}(n_{j}+r)$. Notice that if some coordinates $x_{i}=x_{j}$ for $%
i\neq j,$ then the matrix $(M_{j,k}(x))_{j,k}$ has two identical rows and
therefore $F(x)=0$. Thus the polynomial $P(x)=\prod_{1\leq i<j\leq
q}(x_{i}-x_{j})$ divides $F(x)$. Since both $F$ and $P$ are degree $q(q-1)$
polynomials, it must be that there is a constant $c$ such that for all $x$, 
\begin{equation*}
F(x)=c\prod_{1\leq i<j\leq q}(x_{i}-x_{j}).
\end{equation*}%
In particular, 
\begin{eqnarray*}
F(n_{1}(n_{1}+r),...,n_{q}(n_{q}+r)) &=&\det (M_{j,k})_{j,k} \\
&=&c\prod_{1\leq j<k\leq q}(n_{j}(n_{j}+r)-n_{k}(n_{k}+r)).
\end{eqnarray*}%
Hence, 
\begin{eqnarray*}
\phi _{\lambda }(X) &=&\frac{C\det (M_{j,k})_{j,k}}{%
\prod_{j=1}^{q}P_{n_{j}}^{(p-q,0)}(1)\prod_{1\leq j<k\leq
q}(n_{j}(n_{j}+r)-n_{k}(n_{k}+r))} \\
&\leq &C\prod_{j=1}^{q}(n_{j}+1)^{-p+q}.
\end{eqnarray*}
\end{proof}

We are now ready to complete the proof of the theorem for arbitrary $X$.

\begin{proof}
\lbrack of Theorem (i)] Let $\exp iX\in A\diagdown N_{G_{c}}(K)$ with $%
X=X(t_{1},...,t_{q})$.

First, suppose there is a pair $j,k$ such that $t_{j}\neq \pm t_{k}\func{mod}%
\pi $. In this case, $\cos (2t_{j})\neq \cos (2t_{k})$. Lemmas \ref{e1+e2}
or \ref{e1+e2(2)} (depending on the situation) show that for such $X,$ $\phi
_{\lambda }(X)$ is bounded in modulus by $C\prod_{j=1}^{q-1}(n_{j}+1)^{-1}$.

So we can assume that for every $j,k,$ either $t_{j}\equiv t_{k}\func{mod}%
\pi $ or $t_{j}\equiv -t_{k}\func{mod}\pi $ (or both). If there is some
\thinspace $j$ such that $t_{j}\equiv 0\func{mod}\pi ,$ then this will be
true for all $t_{k}$. Then $\alpha (X)\equiv 0\func{mod}\pi $ for all
restricted roots $\alpha $ and that contradicts the assumption that $\exp
iX\notin N_{G_{c}}(K)$.

If some $t_{j}\equiv \pi /2\func{mod}\pi ,$ then the same is true for all $%
t_{k}$ and then $\cos (2t_{j})=-1$ for all $j$. In the case that $p=q$, as
the positive restricted roots are only of the form $\alpha =e_{j}\pm e_{k}$
and $2e_{j}$, we again have $\alpha (X)\equiv 0\func{mod}\pi $ for all $%
\alpha $, giving a contradiction. If $p>q,$ Lemma \ref{e1} gives the bound 
\begin{equation*}
\left\vert \phi _{\lambda }(X)\right\vert \leq
C\tprod\limits_{j=1}^{q}(n_{j}+r)^{-p+q}\leq
C\tprod\limits_{j=1}^{q}(n_{j}+r)^{-1}.
\end{equation*}

Otherwise we must have $t_{j}\neq 0,\pi /2\func{mod}\pi $ for any $j$, but $%
\cos (2t_{j})=\cos (2t_{k})$ for all $j,k,$ and then Lemma \ref{2e1} says
that $\left\vert \phi _{\lambda }(X)\right\vert \leq
C\prod_{j=1}^{q-1}(n_{j}+1)^{-1/2}$ when $p=q$ and 
\begin{equation*}
\left\vert \phi _{\lambda }(X)\right\vert \leq
C\prod_{j=1}^{q-1}(n_{j}+1)^{-3/2}(n_{q}+1)^{-1}\leq
C\prod_{j=1}^{q}(n_{j}+1)^{-1}
\end{equation*}%
for $p>q$. That completes the proof.
\end{proof}

\section{$L^{2}$ and other smoothness results}

With these decay estimates we can determine when convolution powers of
orbital measures have square integrable density functions.

\begin{theorem}
\label{main} Let $a\in A\diagdown $ $N_{G_{c}}(K)$.

(i) Then $\mu _{a}^{k}\in L^{2}$ provided $k>\max (p,2(p-q)+3)$ if $p>q$ or $%
k>\max (2p,6)$ if $p=q$.

(ii) If $a$ is regular, then $\mu _{a}^{2}\in L^{2}$ when $p>q$ and $\mu
_{a}^{3}\in L^{2}$ when $p=q$.
\end{theorem}

\begin{proof}
We will use the Weyl character formula, which in this setting tells that 
\begin{equation*}
\left\Vert \mu _{a}^{k}\right\Vert _{2}^{2}=\sum_{\lambda }d_{\lambda
}Tr\left\vert \widehat{\mu _{a}}(\lambda \right\vert ^{2k}=\sum_{\lambda \in
\Lambda _{sph}}d_{\lambda }\left\vert \phi _{\lambda }(a)\right\vert ^{2k},
\end{equation*}%
c.f., \cite{AGP}. As we already have estimates on the rate of decay of the
spherical functions, the key additional idea needed to prove this result is
a bound on the growth in the degree of $\lambda $ as a function of $%
(n_{1},...,n_{q})$. This uses the Weyl degree formula which states%
\begin{equation*}
\deg \lambda =C\tprod\limits_{\alpha }<\alpha ,\lambda +\rho >
\end{equation*}%
where $\rho $ is half the sum of the positive roots. Here 
\begin{equation*}
\rho =\sum_{j=1}^{q}\left( \frac{r}{2}+q-j\right) 2e_{j},
\end{equation*}%
so 
\begin{equation*}
\lambda +\rho =\sum_{j=1}^{q}\left( m_{j}+q-j+\frac{r}{2}\right)
2e_{j}=\sum_{j=1}^{q}\left( n_{j}+\frac{r}{2}\right) 2e_{j}\text{.}
\end{equation*}%
Thus (a) $<e_{j}+e_{i},\lambda +\rho >$ $\geq n_{j}+n_{i}+r\geq n_{j}+1$ if $%
j<i$ (and there are $q-j$ such roots, each with multiplicity $2)$;

(b) $<e_{j}-e_{i},\lambda +\rho >$ $\geq n_{j}-n_{i}\geq 1$ if $j<i$;

(c) $<(2)e_{j},\lambda +\rho >$ $\geq n_{j}+1$ (with multiplicity $1$ for
the roots $2e_{j}$ and multiplicity $2(p-q)$ for the roots $e_{j}$).

Consequently, 
\begin{equation}
\deg \lambda \sim
\tprod\limits_{j=1}^{q}(n_{j}+1)^{2(q-j)+1+2(p-q)}=\tprod%
\limits_{j=1}^{q}(n_{j}+1)^{2p-2j+1}.  \label{deg}
\end{equation}%
Thus if $p>q,$ Theorem \ref{decay}(i) and the degree formula (\ref{deg})
tells us that%
\begin{eqnarray*}
\left\Vert \mu _{a}^{k}\right\Vert _{2}^{2} &\leq &C\sum_{\lambda
}\tprod\limits_{j=1}^{q}(n_{j}+1)^{2p-2j+1}\tprod%
\limits_{j=1}^{q-1}(n_{j}+1)^{-2k} \\
&\leq &C\sum_{n_{1}>n_{2}>\cdot \cdot \cdot >n_{q-1}\geq
1}\tprod\limits_{j=1}^{q-1}(n_{j}+1)^{2p-2j+1-2k}%
\sum_{n_{q}<n_{q-1}}(n_{q}+1)^{1+2(p-q)}.
\end{eqnarray*}%
Since $\sum_{n_{q}<n_{q-1}}(n_{q}+1)^{t}\leq C(n_{q-1}+1)^{t+1}$ when $t\geq
0$, it follows that 
\begin{equation*}
\left\Vert \mu _{a}^{k}\right\Vert _{2}^{2}\leq \sum_{n_{1}>n_{2}>\cdot
\cdot \cdot >n_{q-1}\geq
1}\tprod\limits_{j=1}^{q-2}(n_{j}+1)^{2p-2j+1-2k}(n_{q-1}+1)^{4p-4q-2k+5}.
\end{equation*}%
This sum is clearly finite if $2p-2j+1-2k<-1$ for all $j=1,...,q-2$
(equivalently, $p<k$) and $4p-4q-2k+5<-1$ (i.e., $k>2(p-q)+3$), as we
desired to show.

Similar reasoning using the other formula from Theorem \ref{decay}(i) for
the case $p=q$ or the formula from part (ii) in the regular case, gives the
other statements.
\end{proof}

\begin{remark}
We remark that the conclusion $\mu _{a}^{2}\in L^{2}$ for $a$ regular is
sharp since $\mu _{a}$ is singular. Previously, it was shown in \cite{AA}
that when $a$ is regular $\mu _{a}^{k}\in L^{2}$ if 
\begin{equation}
k>\frac{1+\binom{p+q}{2}}{2p-q}.  \label{AA}
\end{equation}
\end{remark}

In \cite{AA}, conditions are given under which the density function of $\mu
_{a}^{k}$ belongs to $C^{d}(G_{c})$. This is done by noting that if $%
H^{s}(G_{c})$ is the Sobolev space of functions in $L^{2}$ with derivatives
up to order $s$ in $L^{2},$ then the Sobolev embedding theorem \cite[Theorem
1.2.1]{Jo}, implies that if $s>d+\dim G_{c}/2,$ then $H^{s}(G_{c})\subseteq
C^{d}(G_{c})$. Moreover, the $H^{s}$ norm can be computed as%
\begin{equation*}
\left\Vert \mu _{a}^{k}\right\Vert _{H^{s}}^{2}=\sum_{\lambda \in \Lambda
_{sph}}d_{\lambda }(1+\kappa _{\lambda })^{s}\left\vert \phi _{\lambda
}(a)\right\vert ^{2k}
\end{equation*}%
where $\kappa _{\lambda }$ is the Casmir constant, $\kappa _{\lambda
}=<\lambda +2\rho ,\lambda >\sim \left\Vert \lambda \right\Vert ^{2}\sim
n_{1}^{2}$. They deduce that $\mu _{a}^{k}\in H^{s}$ for $k>\frac{1+\binom{%
p+q}{2}+2s}{2p-q}$ when $a$ is regular.

With our better estimates on the decay of the spherical functions, we can
obtain the following stronger results.

\begin{proposition}
(i) For any $a\notin N_{G_{c}}(K),$ $\mu _{a}^{k}\in H^{s}$ if $k>\max
(s+p,2(p-q)+3)$ if $p>q$ and $k>\max (2s+2p,6)$ if $p=q$.

(ii) If $a$ is regular, then $\mu _{a}^{k}\in H^{s}$ if $k>(p+s)/(p-1/2)$.
\end{proposition}

The proof involves the same types of calculations as in the proof of the
theorem above.

\begin{remark}
It would be interesting to know if these results are sharp and also to
determine for each $a,$ the minimal exponent $k$ such that $\mu _{a}^{k}\in
L^{2}$ (or $C^{d}).$
\end{remark}

\end{document}